\documentclass{amsart}

\usepackage{microtype}
\usepackage{lmodern}
\usepackage{amsthm}  
\usepackage[linesnumbered, ruled]{algorithm2e}  
\usepackage{mathtools}  
\usepackage[english]{babel}
\usepackage{quiver}  
\usepackage[style=alphabetic, backend=biber, maxbibnames=99, maxcitenames=99]{biblatex}
\usepackage[hidelinks]{hyperref}  

\bibliography{linear}

\newtheorem{theorem}{Theorem}[section]  
\newtheorem{lemma}[theorem]{Lemma}      
\newtheorem{proposition}[theorem]{Proposition}
\newtheorem{corollary}[theorem]{Corollary}
\theoremstyle{definition}
\newtheorem{definition}[theorem]{Definition}

\theoremstyle{remark}
\newtheorem{remark}[theorem]{Remark}

\numberwithin{equation}{section}

\newcommand{\tallvec}[1]{
  \boxed{
    \vphantom{\rule[-2em]{0pt}{5em}} 
    \, #1 \, \cdots 
  }
}

\newcommand{\sqbox}[1]{
  \fbox{
    \parbox[c][1.6cm][c]{1.6cm}{\centering $#1$}
  }
}

\newcommand{\ChainGcd}{\textsc{ChainGcd}}
\newcommand{\LlcmTree}{\textsc{LlcmTree}}
\newcommand{\MorphismKernelInvariants}{\textsc{MorphismKernelInvariants}}
\newcommand{\ModuleOfPointsInvariants}{\textsc{ModuleOfPointsInvariants}}
\newcommand{\TorsionFromModuleOfPoints}{\textsc{TorsionFromModuleOfPoints}}
\newcommand{\MultipointEvaluation}{\textsc{MultipointEvaluation}}

\newcommand{\FunctionName}[5]{
  \[
    \begin{array}{rrcl}
      #1 : & #2 & \to     & #3 \\
           &   #4 & \mapsto & #5
    \end{array}
  \]
}
\newcommand{\FunctionNoname}[4]{
  \[
    \begin{array}{rcl}
      #1 & \to     & #2 \\
      #3 & \mapsto & #4
    \end{array}
  \]
}
\newcommand{\FunctionNameNodef}[3]{
  \[
    \begin{array}{rrcl}
      #1 : & #2 & \to & #3
    \end{array}
  \]
}

\SetEndCharOfAlgoLine{\unskip.}

\newcommand{\PM}{\mathbf{PM}}
\newcommand{\PD}{\mathbf{PD}}
\newcommand{\PG}{\mathbf{PG}}
\newcommand{\EV}{\mathbf{EV}}
\newcommand{\MM}{\mathbf{MM}}

\newcommand{\functiongeq}[1]{{#1}^{\geqslant 1}}
\newcommand{\SM}{\mathbf{SM}}
\newcommand{\SMA}{\mathbf{SM}_\mathrm{A}}
\newcommand{\SMF}{\mathbf{SM}_\mathrm{F}}
\newcommand{\SMgeq}{\functiongeq{\SM}}
\newcommand{\SMAgeq}{\functiongeq{\SMA}}
\newcommand{\SMFgeq}{\functiongeq{\SMF}}

\newcommand{\Otilde}{\widetilde{O}}

\newcommand{\RRii}{[1, \infty[}
\newcommand{\NNs}{\ZZ_{>0}}

\AtBeginDocument{%
  \renewcommand{\Im}{\mathrm{Im}}%
}

\DeclareMathOperator{\llcm}{llcm}
\DeclareMathOperator{\rgcd}{rgcd}

\newcommand{\Ktau}{{K\{\tau\}}}
\newcommand{\M}{{\mathbb M}}

\newcommand{\Fq}{{\mathbb{F}_q}}
\newcommand{\Kbar}{{\overline{K}}}

\newcommand{\NN}{\mathbb{Z}_{\geqslant 0}}
\newcommand{\ZZ}{\mathbb{Z}}

\renewcommand{\d}{\mathfrak d}
\newcommand{\E}{\mathcal E}
\newcommand{\B}{\mathcal B}
\DeclareMathOperator{\Id}{Id}
\DeclareMathOperator{\Fitt}{Fitt}

\newcommand{\MphiT}{M_{\phi_T}}

\DeclareMathOperator{\Fnf}{F}

\newcommand{\stor}[1]{\MM(#1)(\log #1)(\log \log #1)}

\begin{document}


\title[Computing points of general Drinfeld modules]
      {Computing submodules of points of general Drinfeld modules over finite fields}

\author{Antoine Leudière}
\address{University of Calgary}
\email{antoine.leudiere@ucalgary.ca}


\author{Renate Scheidler}
\email{rscheidl@ucalgary.ca}
\address{University of Calgary}


\begin{abstract}

  We present an algorithm for computing the structure of any submodule of the
  module of points of a Drinfeld $A$-module over a finite field, where $A$ is a
  function ring over $\Fq$. When the function ring is $A = \Fq[T]$, we
  additionally compute a Frobenius decomposition of said submodule. Our
  algorithms apply in particular to kernels of isogenies and torsion
  submodules. They are presented within the frameworks of Frobenius normal
  forms, presentations of modules, and Fitting ideals. They rely largely on
  efficient and classical linear algebra methods, combined with fast arithmetic
  of Ore polynomials. We analyze the complexity of our algorithms, explore
  optimizations, and provide an implementation in SageMath. Finally, we compute
  a simple invariant attached to a Drinfeld $\Fq[T]$-module that encodes all
  the polynomials in $\Fq[T]$ whose associated torsion is rational.

\end{abstract}

\maketitle

\section{Introduction}

Drinfeld modules are among the most important tools in function field
arithmetic. Rank two Drinfeld modules bear close resemblance to elliptic curves and the
theory of complex multiplication, rank one Drinfeld modules are analogous to
roots of unity and the Kronecker-Weber theorem, but Drinfeld modules of higher
rank have no direct analogue in zero characteristic. Drinfeld modules were
introduced in 1977 to solve the Langlands conjectures for $\mathrm{GL}_r$ in
function fields \cite{drinfeld_commutative_1977, lafforgue_chtoucas_2002}, with
remarkable success. In the context of more practical applications, Drinfeld modules
can be used for state-of-the-art factorization of polynomials over finite
fields \cite{doliskani_drinfeld_2021} and computer algebra of linear polynomials
\cite{bastioni_characteristic_2025}. More recently, Drinfeld modules were discovered to be extremely well suited to applications in coding theory \cite{bastioni_optimal_2024, micheli_rank_2026}.
In light of the prominent technical and historical role of function fields and
algebraic curves in coding theory \cite{goppa_algebraico-geometric_1983}, Drinfeld modules offer a promising approach for significantly advancing 
the state-of-the-art in algebraic geometry and rank-metric codes.

Despite these successful applications of Drinfeld modules, research on
their computational aspects lags far behind its counterpart for
elliptic curves. The works of Caranay, Greenberg and Scheidler in 2018 were
among the first to explicitly inquire about algorithmic aspects of Drinfeld
modules, focusing only on ordinary rank two Drinfeld modules
\cite{caranay_computing_2018, brenner_computing_2020}. Musleh and Schost subsequently deployed modern computer algebra
techniques to compute the characteristic polynomial of the Frobenius endomorphism for
rank two Drinfeld modules over a finite field  \cite{musleh_computing_2023}. The rank two assumption made it possible to adapt methods from elliptic curves to Drinfeld modules. This rather restrictive assumption was only lifted as recently as 2023 by Musleh and Schost \cite{musleh_computing_2023} and Caruso and Leudière
\cite{caruso_algorithms_2026} with the introduction of entirely
new techniques. Importantly, some of these methods and tools, such as Anderson
motives for computational purposes, have no direct computationally efficient
analogue for elliptic curves. Our work herein continues in this direction. 

In \cite{caruso_algorithms_2026}, Caruso and Leudière computed norms of isogenies
and characteristic polynomials of endomorphisms, the function field equivalent
of point counting for elliptic curves. We build on this work and present an efficient general algorithm for computing the module structure, \emph{via} the invariant factors, of the kernel of
a morphism of Drinfeld modules over a finite field. In particular, the special case of the zero morphism recovers the structure of the full module of rational
points. Our method takes a different approach from prior work; in particular, it is not rooted in elliptic curve machinery. While  relying on classical tools such as Frobenius and Smith normal forms,
presentations of modules, and Fitting ideals, it leverages features specific to
Drinfeld modules and state-of-the-art techniques in computer algebra. Most
importantly, we take advantage of the structure of finitely generated
$\Fq$-algebra of the function ring $A$ of a Drinfeld module. This has no analogue in the realm of elliptic curves, where the role of the function ring is played by the integers~$\ZZ$. Our algorithm applies to  Drinfeld $A$-modules of arbitrary rank over any finite field for any function ring $A$.  

In the case $A = \Fq[T]$, we further give an algorithm for computing Frobenius decompositions of submodules of the module of points (Algorithm~\ref{algo:MorphismKernelInvariants} and provide a precise
asymptotic complexity analysis. Fixing $q$ relative to other
parameters, this algorithm computes the
invariant factors and a Frobenius decomposition of the kernel of a morphism for
a cost of $\Otilde(dr + dn + d^\omega)$, where $n$ is the $\tau$-degree of the
input morphism , $r$ is the rank of its domain Drinfeld module, $d$ is the $\Fq$-dimension of
the base field, and $2\leqslant \omega \leqslant 3$ is a feasible exponent for
matrix multiplication (Proposition~\ref{proposition:cost-keru}). Other precise
complexity statements are given in the article. 

Finally, given a Drinfeld $\Fq[T]$-module over a finite field, we use Anderson
motives to compute the
unique monic element $g\in\Fq[T]$ such that, for $a \in \Fq[T]$, the $a$-torsion is
rational if and only if $a$ divides $g$. The invariant $g$ is the lcm of all
such $a$, and we compute it without prior knowledge of the invariant factors of
the module of rational points. We are unaware of any efficient method to accomplish the same task for elliptic curves.

Our main algorithms are implemented in SageMath \cite{sagemath, MR4624252}; the code is accessible on
GitHub and is accompanied by a \emph{Jupyter} notebook that readers can readily run in any web browser. All code and instructions are available at \url{https://github.com/kryzar/research-drinfeld-submodules/}.

To the best of our knowledge, ours is the first work to address the problem of effectively and efficiently computing
invariant factors and Frobenius decompositions of submodules of points of a Drinfeld module. 
For elliptic curves, the best algorithms to compute invariant factors are
generic abelian group algorithms. For example, the method of Sutherland
\cite{sutherland_structure_2011} proceeds by computing a basis via 
via extracting discrete logarithms on subgroups of prime power order. While discrete
logarithms can be readily computed for Drinfeld modules
\cite{scanlon_public_2001}, adapting the method of Sutherland to this setting would yield an
algorithm significantly that is generally slower than ours, both in terms of asymptotic complexity analysis
and implementation. Moreover, our method has the advantages of being deterministic, easily implemented, and close to optimal.

%

\newpage
The paper is organized as follows.
\begin{enumerate}
  
  \item Section~\ref{section:background} presents the necessary background. We
    briefly introduce Ore polynomials (\S~\ref{back:ore}) and Drinfeld modules
    (\S~\ref{back:drinfeld}). We review the computer algebra of matrices,
    polynomials and Ore polynomials, with an emphasis on complexity models
    (\S~\ref{back:computer}). Finally, we present a general framework for
    computing invariant factors and Frobenius decompositions of finite modules
    over a Dedekind ring that that are uniquely defined by a finite number of
    matrices (\S~\ref{sec:modules-from-matrices}). This framework readily applies to Drinfeld
    modules. In particular, in \S~\ref{subsubsec:case-FqT-abstract} we
    show how the invariant factors and a Frobenius decomposition of the kernel
    of a morphism can be recovered by computing the Frobenius normal form and a
    change-of-basis matrix of the action of the Drinfeld module on this kernel.
    This approach does not extend to the general case, where we need to resort
    to Fitting ideals to recover invariant factors
    (\S~\ref{subsection:general-case}).

  \item Section~\ref{sec:computer-algebra} describes state-of-the-art multipoint evaluation
    (\S~\ref{subsec:computer-ore}) and revisits algorithms related to divisibility chains of polynomials
    (\S~\ref{subsec:div-chains}).

  \item Section~\ref{applications-drin} presents our main algorithms. Once
    again, we distinguish between the case $A = \Fq[T]$
    (\S~\ref{subsec:comput-case-Fq[T]}) and the general case
    (\S~\ref{subsec:comput-general}). Still in the case $A = \Fq[T]$, we present
    alternative methods for computing torsion submodules (\S~\ref{subsubsec:special-torsion}),
    and discuss the efficiency of our algorithms depending on different low
    level primitives (\S~\ref{subsec:discuss}).

  \item Section~\ref{sec:deciding-rational} presents a simple method for
    finding all $a \in \Fq[T]$ such that the $a$-torsion is rational, given a
    Drinfeld $\Fq[T]$-module over a finite field.

\end{enumerate}

\section{Background}
\label{section:background}

\subsection{Invariant factors and Frobenius decompositions}

We begin by defining the quantities that encode the structure of a finite
module over a Dedekind ring. Our goal is to compute these quantities for
submodules of points of Drinfeld modules over finite fields. 

\begin{theorem}
\label{def:inv-frob}

  Let $A$ be a Dedekind domain and $W$ a finite $A$-module. There exist ideals $\d_1, \dots, \d_\ell \subset A$ and elements
  $x_1,\dots, x_\ell \in W$ such that the following hold.

  \begin{enumerate}

    \item The elements generate the module: $W = A x_1 \oplus\cdots\oplus A x_\ell$;
    \item The ideals form a divisibility chain: $\d_1 \mid \cdots \mid \d_\ell$;
    \item For all $i$, $x_i$ has annihilator $\d_i$, \emph{i.e.} $A x_i \simeq
      A/\d_i$;

  \end{enumerate}
  In particular, $W \simeq \prod_{i=1}^n A/\d_i$, and the chain $(\d_1, \dots,
  \d_\ell)$ is unique. We call the ideals $\d_1,\dots,\d_\ell$ the \emph{invariant
  factors} of $W$ and the elements $x_1, \dots, x_\ell$ a \emph{Frobenius
  decomposition} of $W$. 

\end{theorem}

\subsection{Ore polynomials}
\label{back:ore}

Ore polynomials are the main ingredient in the definition of
Drinfeld modules and their isogenies. Throughout, let $\Fq$ be a finite field
with $q$ elements. Let $K$ be a field extension of $\Fq$, and let $\tau$ be the
$\Fq$-linear Frobenius endomorphism of $\Kbar$ defined by $\tau(x) = x^q$.

\begin{definition}
    The set
  \[
    \Ktau \coloneqq \left\{\sum_{i=1}^n x_i\tau^i, \; n \in \NN, \; x_1,\dots, x_n\in
    K\right\}
  \]
 is an $\Fq$-algebra under composition and addition, called the $\Fq$-algebra of \emph{Ore polynomials} (relative to $\Fq$ and
  $K$).

\end{definition}

Although multiplication of Ore polynomials is usually noncommutative, Ore
polynomials retain many key features of classical polynomials. Specifically, we have the
following:

\begin{enumerate}

  \item The \emph{$\tau$-degree} of a nonzero Oe polynomial $f = \sum_{i=1}^n x_i \tau^i \in
    \Ktau$, denoted $\deg_\tau(f)$, is the maximal~$i$ such that $x_i \neq
    0$. The \emph{$\tau$-valuation} is the minimum~$i$ such that $x_i \neq 0$.

  \item The kernel of $f\in \Ktau$ is an $\Fq$-vector space denoted $\ker
    f$. If $f$ is nonzero, then $\dim_\Fq (\ker (f)) \leqslant \deg_\tau(f)$.
    The \emph{roots} of $f$ may live in $\Kbar$, and we write $\ker_K (f)
    \coloneqq \ker (f) \cap K$.

  \item An Ore polynomial $f\in\Ktau$ is said to be \emph{separable} if
    $\dim_\Fq(\ker(f)) = \deg_\tau(f)$ and \emph{inseparable} if
    $\dim_\Fq(\ker(f)) < \deg_\tau(f)$.

  \item \label{item:euclidean} The algebra $\Ktau$ of Ore polynomial is a
    right-euclidean domain for the $\tau$-degree. For any $f, g\in \Ktau$,
    there exist $\alpha,\beta \in \Ktau$ such that $f = \alpha g + \beta$ and
    $\deg_\tau(\beta) < \deg_\tau(g)$.
    Therefore, for any $f,g\in\Ktau$ not both zero, we can define $\rgcd(f, g)$
    as the unique monic Ore polynomial such that $\Ktau f + \Ktau g = \Ktau
    \rgcd(f, g)$, and we define $\llcm(f, g)$ to be the unique monic Ore polynomial such
    that $\Ktau f \cap \Ktau g = \Ktau \llcm(f, g)$.

\end{enumerate}

\subsection{Drinfeld modules}
\label{back:drinfeld}

We now turn to the definition of Drinfeld modules and related objects, including isogenies and modules of points. The reader is referred to \cite{papikian_drinfeld_2023} or
\cite{armana_computational_2025} for a general introduction to the subject. We begin by
establishing notation and terminology. Throughout, let $C$ be a smooth projective geometrically
connected curve over $\Fq$ with a geometric closed point $\infty$. Let $A$ be
the ring of functions of $\Fq(C)$ that are regular outside $\infty$. Then $A$ is a
Dedekind domain called a \emph{function ring}. Let $K$ be a field extension of
$\Fq$, equipped with a morphism of $\Fq$-algebras $\gamma: A \to K$. The pair
$(K, \gamma)$ is called an \emph{$A$-field}. Let $\Kbar$ be an algebraic
closure of $K$.

\begin{definition}
\label{definition:dm}

  A \emph{Drinfeld $A$-module over $(K, \gamma)$} is a morphism of
  $\Fq$-algebras
    \FunctionName
    {\phi}
    {A}
    {\Ktau}
    {a}
    {\phi_a,}
  such that $\Im(\phi) \not\subset K$ and for each $a\in A$, the constant
  coefficient of $\phi_a$ is $\gamma(a)$.

\end{definition}

Drinfeld modules, despite their names, are not modules, but $\Fq$-algebra
homomorphisms. However, they give rise to module structures on extensions
$L/K$.

\begin{definition}
\label{definition:module-of-points}

  For any extension $L/K$, the \emph{module of $L$-points of a Drinfeld module $\phi$}, denoted $\phi(L)$, is the
  $A$-module defined by
  \FunctionNoname
    {A \times L}
    {L}
    {(a, z)}
    {\phi_a(z).}
  The elements of $\phi(K)$ are called \emph{rational points}.
\end{definition}


The underlying set of $\phi(L)$ is the same for all Drinfeld modules, but its structure as an $A$-module depends on $\phi$. Our goal in this paper
is to compute the invariant factors of any $A$-submodule of $\phi(K)$, when $K$
is finite. When $A = \Fq[T]$, we also give a Frobenius decomposition of any
submodule of $\phi(L)$ (Algorithm~\ref{algo:MorphismKernelInvariants}). To represent these
submodules, we introduce morphisms and isogenies.

\begin{definition}

  Let $\phi, \psi$ be two Drinfeld $A$-modules over $(K,\gamma)$. A
  \emph{morphism from $\phi$ to $\psi$} is an Ore polynomial $u \in \Ktau$ such that
  $u\phi_a = \psi_a u$ for every $a \in A$. An \emph{isogeny} is a nonzero
  morphism.

\end{definition}

The kernel of a morphism $u$ with domain $\phi$ is a finite $A$-submodule of
$\phi(\Kbar)$. To avoid ambiguity, we write $\phi(\ker(u))$ for $\ker(u)$ when viewed 
as an $A$-module. Conversely, every finite $A$-submodule of $\phi(K)$ is
the kernel of an isogeny starting from $\phi$ that can be effectively computed
\cite[Proposition~3.2]{armana_computational_2025}. Therefore, our aim is to
compute invariant factors and Frobenius decompositions of kernels of morphisms.

An important invariant attached to every Drinfeld module is its \emph{rank}.

\begin{lemma}
\label{def:rank}

  Let $\phi$ be a Drinfeld $A$-module over $(K, \gamma)$. There exists a unique
  integer $r \geqslant 1$ such that for all $a \in A$, $\deg_\tau(\phi_a) = r
  \dim_{\Fq}(A/aA)$. The integer $r$  is called the \emph{rank} of $\phi$.

\end{lemma}

\begin{remark}

  Defining a Drinfeld $A$-module amounts to defining an injection $A\to\Ktau$.
  It is not clear how to do so, and such injections may not exist at all
  \cite[Remark~4.3]{papikian_drinfeld_2023}. On the other hand, the case $A =
  \Fq[T]$ (corresponding to $C = \mathbb P^1_{\Fq}$) is straightforward: a
  Drinfeld $\Fq[T]$-module $\phi$ is uniquely defined by the image~$\phi_T$ of~$T$, and its rank is simply $\deg_\tau(\phi)$.

\end{remark}

Finally, a particular topic of interest, for example for computing Tate
modules, Weil pairings and, more recently, codes in rank metric
\cite{bastioni_optimal_2024, micheli_rank_2026}) is the computation of the
kernel of $\phi_a$ for $a \in A$.

\begin{definition}

  Let $a\in A$. The \emph{$a$-torsion} of $\phi$ is the kernel of the morphism~$\phi_a$. It is an $A$-submodule of $\phi(\Kbar)$, denoted $\phi[a] \coloneqq
  \phi(\Kbar)[a]$. The $a$-torsion of $\phi$ contained in $L$ is $\phi(L)[a]
  \coloneqq \phi[a]\cap\phi(L)$. The $a$-torsion of $\phi$ is said to be
  \emph{$L$-rational} if $\phi(L)[a] = \phi[a]$.

\end{definition}

%
%

%
%

\subsection{Computer algebra}
\label{back:computer}

We now specify our complexity model and recall classical computer algebra
reductions for polynomials and matrices in \S~\ref{subsubsec:cost-pol-mat}. We
particularly emphasize fast primitives for Ore polynomials
(\S~\ref{subsubsec:cost-ore}, \S~\ref{subsubsec:divide-and-conquer}).

\subsubsection{Cost functions for polynomials and matrices}
\label{subsubsec:cost-pol-mat}

We start by defining cost functions. Let $F$ be a field. Two polynomials in
$F[T]$ with degrees $\leqslant n$ can be multiplied for a cost of $O(\PM(n))$
operations in $F$. We assume that $\PM$ is \emph{super-additive}, \emph{i.e.} $\PM(x+y) \leqslant \PM(x) + \PM(y)$ for all inputs $x, y$. When $F$ is a
finite field, we can pick $\PM(n) = n \log n \log \log n$
\cite[\S~8.3]{von_zur_gathen_modern_2013} (see also \cite{10.1145/3005344} for
a slightly better complexity). Moreover:

\begin{itemize}

  \item The Euclidean division of two polynomials in $F[T]$ with respective
    degrees $m$ and $n$, $m \geqslant n$, can be performed for a cost of
    $O(\PD(m, n))$ operations in $F$. We can pick $\PD(m, n) = O(\PM(m - n) +
    \PM(n))$ \cite[\S~11]{von_zur_gathen_modern_2013}.


  \item The gcd of two polynomials with degrees $\leqslant n$ can be computed
    for a cost of $O(\PG(n))$ operations in $F$. We can pick $\PG(n) = \PM(n)
    \log n$ \cite[\S~11]{von_zur_gathen_modern_2013} and assume that
    $\PG(n) \in O(n^2)$.

\end{itemize}

Next, we introduce cost functions for matrix arithmetic. The multiplication of two matrices in $F^{d\times d}$
can be performed for a cost of $O(\MM(d))$ operations in $F$. In particular, we
let $2 \leqslant \omega \leqslant 3$ be such that $\MM(d) = d^\omega$.

\subsubsection{Ore polynomials}
\label{subsubsec:cost-ore}

Almost all our algorithms rely on low level primitives on Ore polynomials:
Euclidean division, llcm and rgcd computation, and multipoint evaluation. Let
$K$ be a finite extension of $\Fq$ with $[K:\Fq] = d$. The state of the art for
computer algebra of Ore polynomials in $\Ktau$ is due to Caruso and Le Borgne
\cite{caruso_new_2017}. One issue with their computation model is the
assumption that all operations in $K$, including applications of the Frobenius
endomorphism (\emph{i.e.} computing $x^q$ for $x\in K$), can be performed for a
cost of $\Otilde(d)$ operations in~$\Fq$. This assumption does not always hold (see \cite{musleh_computing_2023} for a discussion on this topic).
Therefore, we count arithmetic operations and applications of the Frobenius
endomorphism separately.

\begin{definition}

  We let
  \FunctionNameNodef
    {\SM = (\SMA, \SMF)}
    {\NNs^2}
    {\RRii}
  be a function such that two Ore polynomials with degrees $\leqslant n$ can be
  multiplied for a cost of $O(\SMA(n))$ arithmetic operations in $K$ and
  $O(\SMF(n))$ Frobenius applications. We refer to this cost as \emph{$O(\SM(n))$
  arithmetic and Frobenius operations in~$K$}.

\end{definition}

\subsubsection{Divide and conquer algorithms}
\label{subsubsec:divide-and-conquer}

Later, we describe the fast multipoint evaluation algorithm of Ore polynomials
(\S~\ref{subsec:computer-ore}) using \emph{divide and conquer} algorithms. We
recall how to analyse the cost of such algorithms. On an input of size $\ell$, the
cost of a divide and conquer algorithm is given by a recursion of the form 
\begin{equation}
\label{eq:recursion}
  T(\ell) \leqslant 2 T\left(\frac \ell 2\right)
       + \alpha\left(\left\lceil\frac \ell 2\right\rceil\right),
\end{equation} 
If $\alpha$ is super-additive, we can solve the recursion:
\begin{equation}
\label{eq:divide-conquer}
  T(\ell) = O(\alpha(\ell) \log \ell).
\end{equation}
It is therefore convenient to introduce the following notation:
\begin{proposition}
\label{proposition:geqslant}

  For $\alpha:\NNs \to \RRii$, we define $\alpha^{\geqslant 1}: \NNs \to \RRii$ via

  \[
    \alpha^{\geqslant 1}(\ell) \coloneqq \ell \sup_{1\leqslant \ell' \leqslant \ell} \frac
  {\alpha(\ell')}{\ell'}.
  \]
  The function $\alpha^{\geqslant 1}$ is super-additive and is the minimal such
  function exceeding $\alpha$.

\end{proposition}
In our case, we do not assume $\SM$ to be super-additive, as the value for
$\SM$ given by Caruso and Le Borgne does not satisfy that
property~\cite[p.~83]{caruso_fast_2017}. Following their methodology, we will
instead use $\SMgeq \coloneqq (\SMAgeq, \SMFgeq)$ to solve recursions coming
from divide and conquer algorithms.

\subsection{Modules from matrices}
\label{sec:modules-from-matrices}

We now present a roadmap for computing
invariant factors and Frobenius decompositions of submodules of points of
Drinfeld modules over finite fields.

\subsubsection{The case $A = \Fq[T]$}
\label{subsubsec:case-FqT-abstract}

This is the simplest case, as for a Drinfeld $\Fq[T]$-module
$\phi$, $\phi(K)$ is uniquely determined by the action of $\phi_T$. Upon fixing
an $\Fq$-basis of the base field, $\phi_T$ is represented by a matrix that
contains all the information about~$\phi(K)$. The task is therefore to extract
information from this matrix. For clarity of presentation, we describe a general
framework using vector spaces and matrices that readily applies to Drinfeld modules. 

Let $F$ be a field, 
$V$ an $F$-vector space, and $\varphi$ an endomorphism of $V$ (in our
context, the role of $\varphi$ is played by $\phi_T$). This set-up canonically
yields an $F$-module structure on $V$:
\FunctionNoname
  {F[T] \times V}
  {V}
  {(a(\varphi), x)}
  {a(\varphi)(x).}
This new module is denoted $V[\varphi]$.

\begin{remark}

  For a Drinfeld $\Fq[T]$-module $\phi$ over a finite field $K$ and a morphism~$u$ whose domain is $\phi$, this situation corresponds to taking $F = \Fq$,
  $V = \ker_K u$, $\varphi = \phi_T$ and $V[\varphi] = \phi(K)$. The goal is
  therefore to compute the invariant factors and a Frobenius decomposition of
  $V[\varphi]$.

\end{remark}

If $V$ has finite $F$-dimension, we can endow it with an $F$-basis $\mathcal E
= (\varepsilon_1, \dots, \varepsilon_d)$, and $\varphi$ can be represented by a
matrix $M \in F^{d\times d}$. The invariant factors of $V[\varphi]$ can be
retrieved by computing the \emph{Frobenius normal form} of the matrix $M$: a
block diagonal matrix of the form
\[
  \Fnf_M =
  \begin{pmatrix}
    \boxed{C_{d_1}} & & \\
    & \ddots & \\
    & & \boxed{C_{d_\ell}}
  \end{pmatrix}
  \in F^{d\times d}
\]
such that
\begin{enumerate}
  \item each $C_{d_i}$ is the companion matrix of a polynomial $d_i \in
    F[T]$;
  \item the polynomials form a divisibility chain: $d_1 | \cdots | d_\ell$;
  \item the marices $M$ and $\Fnf_M$ are similar over $F$.
\end{enumerate}
If furthermore $S$ is a change-of-basis matrix such that
\[
  \Fnf_M = S^{-1} M S,
\]
then one immediately recovers a Frobenius decomposition of $V[\varphi]$ from 
$S^{-1}$. We write $S^{-1} \in F^{d\times d}$ as a horizontal block
matrix in which the $i$-th block has $\deg(d_i)$ columns, with the $i$-th
column labeled $x_i$:
\[
  S^{-1} = 
  \begin{pmatrix}
    \tallvec{x_1} & \cdots & \tallvec{x_\ell}
  \end{pmatrix}
  \in F^{d\times d}
\]

\begin{proposition}
\label{proposition:inv-frob-Fq[T]}

  With the above notation, the ideals $(d_1), \dots, (d_\ell)$ are the invariant
  factors of $V[\varphi]$, and the elements $x_1, \dots, x_\ell$ form a Frobenius
  decomposition of $V[\varphi]$.

\end{proposition}

By Proposition~\ref{proposition:inv-frob-Fq[T]}, to recover the invariant
factors and a Frobenius decomposition of $V[\varphi]$, it suffices to represent
$\varphi$ as a matrix and compute its Frobenius normal form. To compute
Frobenius normal forms, one can use an algorithm of Storjohann
\cite[Chapter~9]{storjohann_algorithms_2000}.

\begin{lemma}
\label{lemma:cost-frob}

  Let $M \in F^{d\times d}$ be a matrix. The Frobenius normal
  form of $M$, along with a unimodular change-of-basis matrix, can be computed for a cost of
  $O(\stor{d})$ operations in $F$.

\end{lemma}

Let $\phi$ be a Drinfeld $\Fq[T]$-module over a finite field $K$, and let $u$
be a morphism whose domain is $\phi$. One can therefore compute the invariant
factors and a Frobenius decomposition of $\phi(\ker_K(u))$ (an $\Fq[T]$-module)
by computing the Frobenius normal form of the matrix of $\phi_T$ relative to an
$\Fq$-basis of $\ker_K(u)$ (an $\Fq$-vector space), and extracting the invariant factors as described above.
In \S~\ref{subsec:computer-ore}, we revisit fast multipoint evaluation of Ore
polynomials in order to efficiently compute matrices of Ore polynomials, and
our main algorithm (Algorithm~\ref{algo:MorphismKernelInvariants}) is presented
after.

\subsubsection{The general case}
\label{subsection:general-case}

The difficulty of the general case lies in the fact that a function ring cannot
be assumed to be generated by one element. To overcome this difficulty, we use presentations of modules.

\begin{definition}
  Let $A$ be a ring and $W$ an $A$-module. A \emph{presentation} of $W$
  is an exact sequence of $A$-modules
\[\begin{tikzcd}[ampersand replacement=\&]
	{W'} \& {W''} \& W \& 0.
	\arrow["\alpha", from=1-1, to=1-2]
	\arrow["\beta", from=1-2, to=1-3]
	\arrow[from=1-3, to=1-4]
\end{tikzcd}\]
\end{definition}

\begin{remark}

  This definition is motivated by the fact that the map $\alpha$ uniquely defines $W$ up to isomorphis, as $W \simeq W''/\Im(\alpha)$. In
  particular, if $\alpha$ is represented by a diagonal matrix with entries
  $d_1, \dots, d_m$, then $W \simeq A/(d_1) \times \cdots \times A/(d_\ell)$.

\end{remark}

Once again, we provide a more general treatment that readily applies to Drinfeld
modules. As in \S~\ref{subsubsec:case-FqT-abstract}, le $F$ be a field and $V$ an $F$-vector space of finite dimension $d$. Let $\E =
(\varepsilon_1,\dots,\varepsilon_d)$ be an $F$-basis of $V$. Here,
we do not assume that $V$ comes with a distinguished endomorphism, but rather
that it extends to an $A$-module, where $A$ is an $F$-algebra with a finite
number of generators $g_1, \dots, g_m \in A$. The $A$-module structure is
denoted $W$. Finally, we write $V_A \coloneqq A \otimes_F V$.

\begin{remark}
  The setting described above applies to $F = \Fq$, $V = \ker_K(u)$, $W = \phi(\ker_K(u))$, where $\phi$ is a Drinfeld $A$-module over a
  finite field $K$ with a function ring $A$. Note that every function ring $A$ is finitely generated over $\Fq$. 
\end{remark}

We have a finite presentation
\[\begin{tikzcd}[ampersand replacement=\&]
	{(V_A)^n} \& {V_A} \& W \& {0,}
	\arrow["\chi", from=1-1, to=1-2]
	\arrow["\pi", from=1-2, to=1-3]
	\arrow[from=1-3, to=1-4]
\end{tikzcd}\]
where we define
\[
  \begin{cases}
    \chi\left((a_k \otimes v_k)_{1\leqslant k \leqslant n}\right)
  = \sum_{k=1}^n \left( a_kg_k \otimes v_k - a_k\otimes (g_k\cdot
  v_k)\right),\\
    \pi(a\otimes v) = a\cdot v.
  \end{cases}
\]
The $F$-basis $\E$ of $V$ canonically yields bases of $V_A$ and $V_A^n$, and
$\chi$ is represented by a matrix $M_\chi$ which is a block matrix of companion
matrices
\begin{equation}
\label{eq:stacking}
  M_\chi = 
  \begin{pmatrix}
    \sqbox{g_1 \Id - M_1} \\
    \vdots \\
    \sqbox{g_m \Id - M_m}
  \end{pmatrix}
  \in A^{d\times nd}
\end{equation}
where each $M_i$ is the matrix of the $F$-linear endomorphism $x \mapsto g_i
x$. WHat is $g_i$? Are these $\Fq$-generators of $A$?

\begin{remark}

  In the case $A = \Fq[T]$, we see that $M_\chi$ is the characteristic matrix $T\Id - M$, and thus that $V[\varphi] \simeq A^d/\Im(T\Id - M)$.
  Therefore, the invariant factors of $V[\varphi]$ are exactly the invariant
  factors of the characteristic matrix $T\Id - M$, which are the nonzero
  diagonal elements of its Smith normal form. However, computing the Smith
  normal form of a polynomial matrix is significantly more expensive than
  computing the Frobenius normal form of a matrix of the same size with entries
  in a field.

\end{remark}

Returning to the general module setting, the invariant factors of $W$ are the invariant factors of the matrix
$M_\chi$, which can be retrieved using Fitting ideals.
\cite{Fitting1936}:

\begin{lemma}
\label{lemma:fitting}

  Let
    \[\begin{tikzcd}[ampersand replacement=\&]
    	{W'} \& {W''} \& W \& 0
    	\arrow["\alpha", from=1-1, to=1-2]
    	\arrow["\beta", from=1-2, to=1-3]
    	\arrow[from=1-3, to=1-4]
  \end{tikzcd}\]
  be a presentation of $W$, and assume that $\alpha$ is given (in some
  basis) by a matrix $M\in A^{r\times s}$ with $r, s \in \NN$. For
  $1\leqslant 0 \leqslant d$, let
  \[
    \Fitt_i(W)
  \]
  be the ideal generated by the $(d-i)$-by-$(d-i)$ minors of $M$. The value of
  $\Fitt_i(W)$ does not depend on the choice of presentation.
  For $0\leqslant i \leqslant d$, we call $\Fitt_i(W)$ the \emph{$i$-th Fitting
  ideal of $W$}.

\end{lemma}

A classical result \cite[\S~15.8]{stacks} asserts that the Fitting ideals form a divisibility chain
\[
  \Fitt_d(W) \mid \cdots \mid \Fitt_0(W),
\]
and we have the useful following lemma \cite[Lemma~15.8.4]{stacks}.
\begin{lemma}
\label{lemma:inv-fit}

  Let $\d_1,\dots, \d_\ell$ be the invariant factors of $W$. Then
  \[
    \d_i = \frac {\Fitt(d-i)} {\Fitt(d+1-i)}.
  \]

\end{lemma}

Let $\phi$ be a Drinfeld $A$-module over a finite field, with a function ring $A$, and let $u$ be a morphism whose domain is $\phi$. Let
$g_1,\dots, g_m$ be generators of $A$ as an $\Fq$-algebra. To compute the
invariant factors of $\phi(\ker_K(u))$ (an $A$-module), it therefore suffices
to compute the matrices $M_1, \ldots , M_m$ of the actions of $\phi_{g_1},\dots,\phi_{g_m}$
on $\ker_K(u)$ (an $\Fq$-vector space), compute the matrices $g_i \Id - M_i$ from
these matrices and stack them into the matrix $M_\chi$, compute the Fitting
ideals of $M_\chi$, and finally apply Lemma~\ref{lemma:inv-fit} to obtain the invariant
factors of $\phi(\ker_K(u))$.

\section{Computer algebra}
\label{sec:computer-algebra}

In \S~\ref{sec:modules-from-matrices}, we developed the main theoretical
tools required to compute the invariant factors of an $A$-submodule of $\phi(K)$
when $K$ is finite. It remains to analyze the computational complexity of the underlying algorithms. We revisit
Ore polynomials (\S~\ref{subsec:computer-ore}) and
divisibility chains of polynomials (\S~\ref{subsec:div-chains}).

\subsection{Computer algebra of Ore polynomials}
\label{subsec:computer-ore}

We develop an efficient method for multipoint evaluation of Ore polynomials. First, we
need to assess the cost of fast Euclidean division, llcm and rgcd.
Ore polynomial analogues of the classical methods were already described by
Caruso and Le Borgne (see \cite[\S~3.2.1]{caruso_new_2017} and
\cite[\S~3.2.2]{caruso_fast_2017}). We state their complexity to fit with our
complexity model.

\begin{lemma}
\label{lemma:cost-rgcd-llcm}

  Let $f, g\in\Ktau$ be two Ore polynomials with degrees $\leqslant n$. The
  right-Ore Euclidean division, the rgcd and the llcm of $f$ by $g$ can all be
  performed for a cost of $O(\SMgeq(n) \log n)$ arithmetic and Frobenius
  operations in $K$.

\end{lemma}

Next, we discuss multipoint evaluation. Let $f$ be an Ore polynomial of degree~$n$, and let $X = (x_1, \dots, x_\ell)$ be elements of $K$, with $n \leqslant
\ell$. We wish to efficiently compute $f(x_1), \dots, f(x_\ell)$. For any $x \in K$, set
$L_x \coloneqq \tau - \frac {\tau(x)} x \in\Ktau$. If $R_x \in K$ is the
remainder in the right-division of $f$ by $L_x$, then $f(x) = x R_x$. In other
words, $f(x)$ is the evaluation of $R_x$, viewed as a constant Ore polynomial, at
$x$. For elements $x_{i_1}, \dots, x_{i_{\ell'}} \in X$, we set
\[
  L_{x_{i_1}, \dots, x_{i_{\ell'}}} \coloneqq \llcm(L_{x_{i_1}}, \dots, L_{x_{i_{\ell'}}}).
\]
Let $R_{x_{i_1},\dots, x_{i_{\ell'}}}$ be the remainder
in the right-division of $f$ by $L_{x_{i_1},\dots, x_{i_{\ell'}}}$. Then for
any $x_{i_j}$ we have
\[
  f(x_{i_j}) = R_{x_{i_1},\dots, x_{i_{\ell'}}}(x_{i_j}).
\]
With these identities, we can divide the problem of multipoint evaluation as
follows. Assume for clarity of presentation that $\ell$ is even. We
split our inputs into two halves: $X_l = (x_1,\dots, x_{\frac \ell 2})$ and
$X_r = (x_{\frac \ell 2 + 1}, \dots, x_\ell)$. Provided that we know $L_{X_l}$
and $L_{X_r}$, we obtain $f(x_1),\dots, f(x_\ell)$ by computing $R_{X_l}$ and
$R_{X_r}$, and then $R_{X_l}(x_i)$ for all $x_i \in X_l$ as well as
$R_{X_r}(x_i)$ for all $x_i \in X_r$. The cost of this step is two Euclidean
divisions of Ore polynomials of degrees $\leqslant \frac \ell 2$ and two
recursive calls on an entry of size $\leqslant \frac \ell 2$.

In Algorithm~\ref{algo:tree}, following the classical method, we pre-compute the llcm of $X$ recursively as a binary tree, called the \emph{llcm tree of $X$}, .

\begin{algorithm}[ht]
\caption{\LlcmTree}
\label{algo:tree}

  \KwIn{A subfamily $X'$ of $X\coloneqq(x_1,\dots, x_{\ell})$}
  \KwOut{The llcm tree of $X'$}

  \If{$X' = \{x_i\}$ is a singleton}{
    Return the tree whose only leaf is $L_{x_i}$.
  }
  \Else{
    Split $X'$ into two halves $X'_l$ and $X'_R$ \;
    Compute and set $T_l = \LlcmTree(X'_l)$ \;
    Compute and set $T_r = \LlcmTree(X'_r)$ \;
    Compute the llcm $L$ of the roots of $T_l$ and $T_r$ \;
    \KwRet{the binary tree rooted at $L$ with subtrees $T_l$ and $T_r$\;}
  }

\end{algorithm}

Let $T_{X}$ be the tree generated by Algorithm~\ref{algo:tree}, which we
assume precomputed. We obtain Algorithm~\ref{algo:multieval} for multipoint
evaluation.

\begin{algorithm}[ht]
\caption{\MultipointEvaluation}
\label{algo:multieval}

  \KwIn{An Ore polynomial $f$, elements $x_1,\dots, x_\ell\in K$ and the tree
  $T_{X}$}
  \KwOut{The evaluations $f(x_1), \dots, f(x_\ell)$}

  \If{If $\ell = 1$}{
    \KwRet{$f(x_1)$}
  }
  \Else{

    Split $x_1,\dots, x_\ell$ into two halves $X_l$ and $X_r$ (padding if necessary).

    Let $\mathcal T_l$ be the left subtree of $T_{X}$, and get $L_{X_l}$ as its
    root \;

    Let $\mathcal T_r$ be the right subtree of $T_{X}$, and get $L_{X_r}$ as its
    root \;

    Compute $R_{X_l}$ as the Euclidean division of $f$ by $L_{X_l}$ \;
    Compute $R_{X_r}$ as the Euclidean division of $f$ by $L_{X_r}$ \;

    Compute $f(x_1), \dots, f(x_{\ell/2}) = \MultipointEvaluation(R_{X_l}, X_l, \mathcal T_l)$ \;
    Compute $f(x_{\ell/2+1}), \dots, f(x_{\ell}) = \MultipointEvaluation(R_{X_r}, X_r, \mathcal T_r)$ \;
    \KwRet{$f(x_1), \dots, f(x_\ell)$}
  }

\end{algorithm}

\begin{lemma}
\label{lemma:cost-multipoint}

  Let $x_1,\dots, x_\ell \in K$ and let $f$ be an Ore polynomial of degree
  $\leqslant \ell$. The values $f(x_1),\dots, f(x_\ell)$ can be computed for a cost of
  $O(\SMgeq(\ell)(\log \ell)^2)$ arithmetic and Frobenius operations in $K$.

\end{lemma}

\begin{proof}

  The cost of this task is the cost of sequentially running
  Algorithm~\ref{algo:tree}, followed by Algorithm~\ref{algo:multieval}. For
  both these divide and conquer algorithms, the cost of an input of size
  $\ell$ is given by a recursion of the form~\eqref{eq:recursion}, with $\alpha(\ell) = \SMgeq(\ell)
  \log \ell$
  (Lemma~\ref{lemma:cost-rgcd-llcm}). We conclude the proof by applying
  Equation~\eqref{eq:divide-conquer} and the super-additivity of $x \mapsto
  \SMgeq(x) \log x$ (Proposition~\ref{proposition:geqslant}).
\end{proof}

\begin{remark}

  The cost of Ore polynomial multipoint evaluation is
  $O(\SMgeq(\ell) (\log \ell)^2)$ arithmetic and Frobenius operations, whereas
  the cost of classical multipoint evaluation is $O(\PM(\ell) \log \ell)$
  arithmetic operations \cite[\S~10.1]{von_zur_gathen_modern_2013}. The extra
  $(\log \ell)$ factor comes from the fact that we have to use llcm's instead
  of products.


\end{remark}

\subsection{Computer algebra for divisibility chains}
\label{subsec:div-chains}

We now turn to efficient computations involving on divisibility chains, which
we need to compute the invariant factors and a Frobenius decomposition of
torsion submodule when we know these quantities for the full module of rational
points (\S~\ref{subsubsec:special-torsion}). Let $F$ be a field,  $a \in F[T]$
a polynomial, and $d_1 \mid \cdots \mid d_\ell$ a divisibility chain of $\ell$
polynomials in $F[T]$. Algorithm~\ref{algo:ChainGcd} efficiently computes the
polynomials $\gamma_i \coloneqq \gcd(a, d_i)$ and $\rho_i \coloneqq
d_i/\gamma_i$ for all $1\leqslant i \leqslant \ell$. Later on, we wish to
evaluate these polynomials on vectors and matrices. Let $\delta \coloneqq
\deg(d_1) + \cdots + \deg(d_\ell)$. 

\begin{algorithm}[ht]
\caption{\ChainGcd}
\label{algo:ChainGcd}

  \KwIn{The polynomials $a$ and $d_1, \dots, d_{\ell}$.}
  \KwOut{The lists $(\gamma_1, \dots, \gamma_\ell)$ and
  $(\rho_1,\dots,\rho_\ell)$.}

  Compute $a' = a \mod{d_\ell}$ \; \label{step:redmod}
  Compute $\gamma_\ell = \gcd(a', d_\ell)$ \;
  Compute $\rho_\ell = d_\ell/\gamma_\ell$ \;
    
  \For {$i$ ranging down from $\ell-1$ to $1$}{
    Compute $\gamma_i = \gcd(\gamma_{i+1}, d_i)$ \;
    Compute $\rho_i = d_i / \gamma_i$ \;
  }

  \KwRet{$(\gamma_i, \rho_i)$ and $(\gamma_\ell, \rho_\ell)$.}

\end{algorithm}

\begin{lemma}
\label{lemma:chain-gcd}

  Algorithm~\ref{algo:ChainGcd} computes $(\gamma_1, \dots, \gamma_\ell)$ and
  $(\rho_1, \dots,\rho_\ell)$ for a cost of $O(\PM(\deg(a)) + \PM(\delta)\log
  \delta)$
  operations in $F$, where $\delta = \sum_{i=1}^\ell \deg(d_i)$.

\end{lemma}

\begin{proof}
  All operations are arithmetic operations in $F$.
  Step~\ref{step:redmod} costs $O(\PM(\deg(a) - \deg(d_\ell)) + \PM(\deg(a)))$
  operations. The two subsequent steps cost $O(\PG(\deg(a') + \deg(d_\ell)) +
  \PD(\deg(d_\ell) + \deg(\gamma_\ell)))$ operations, which is
  $O(\PG(\deg(d_\ell)))$. The rest amounts to $\sum_{i=1}^{\ell-1}
  O(\PG(\deg(d_{i+1})) + \PD(\deg(d_i)))$ operations, which is $O(\PG(d))$
  operations. We conclude the proof using the fact that $\PG(\delta) =
  \PM(\delta)\log \delta$.
\end{proof}

Next, we describe how to evaluate divisibility chains on matrices and vectors. Let $M$ be a
$d$-by-$d$ matrix with entries in $F$, and let $x_1, \dots, x_\ell$ be vectors
of $F^d$. We are interested in computing the vectors $y_i = \rho_i(M) x_i$ for $1\leqslant i \leqslant \ell$. Once the polynomials $(\rho_i)_{1 \leqslant
i \leqslant \ell}$ are known (see Lemma~\ref{lemma:chain-gcd} for the cost of
their computation), there are different possibilities to compute
$(y_i)_{1\leqslant i \leqslant \ell}$:

\begin{itemize}
  
  \item We can na\"ively compute $(y_i)_{1\leqslant i \leqslant \ell}$ for
    a cost of $O(\delta d^2)$ operations in $F$.

  \item In our applications, we have $\delta \leqslant d$. We can thus compute
    the matrices $(\rho_i(M))_{1\leqslant i \leqslant \ell}$ sequentially, for
    a total cost of $O(\ell \stor{d})$ operations in $F$, using an algorithm of
    Storjohann \cite[Chapter~9]{storjohann_algorithms_2000}. Then we compute
    $(y_i)_{1\leqslant i \leqslant \ell}$ for the same asymptotic cost.

\end{itemize}

To account for multiple possibilities for computing $(y_i)_{1\leqslant i
\leqslant \ell}$, we introduce a cost function.
\begin{definition}
\label{def:ev}

  Let
  \FunctionNameNodef
    {\EV}
    {\NNs}
    {\RRii}
  be a function such that the quantities $d'_1(M)x_1, \dots, d'_{\ell'}(M)x_{\ell'}$ can be
  computed for a cost of $O(\EV_d(\ell',\delta'))$ operations in $F$, for any
  matrix $M\in F^{d'\times d'}$, any set of vectors $x_1,\dots, x_{\ell'} \in F^d$, and
  any divisibility chain $d'_1 \mid \dots \mid d'_\ell \in F[T]$, where $\delta' =
  \sum_{i=1}^{\ell'} \deg(d'_i) \leqslant d'$.

\end{definition}

\section{Application to Drinfeld modules}
\label{applications-drin}

We now have all the ingredients for computing the structure
of submodules of points of Drinfeld modules over finite fields. More precisely,
we let $(K, \gamma)$ be an $A$-field with $d = \dim_\Fq(K)$, and we fix an
$\Fq$-basis $\mathcal E = (\varepsilon_1, \dots, \varepsilon_d)$ of~$K$. Let~$\phi$ be a Drinfeld $A$-module over $(K, \gamma)$ with rank $r$. Our goal is
to compute the invariant factors of $\phi(\ker_K(u))$ for any morphism $u$
whose domain is $\phi$. Note that this includes~$\phi(K)$, as this is
the kernel of the zero morphism. In the case $A = \Fq[T]$, we also compute a
Frobenius decomposition of any $\Fq[T]$-submodule of $\phi(K)$ (presented as
the kernel of an isogeny), and give a detailed complexity analysis of our algorithms. This 
problem is a direct instance of the framework presented in
\S~\ref{sec:modules-from-matrices}. We first consider the case $A = \Fq[T]$
(\S~\ref{subsec:comput-case-Fq[T]}), as it can take advantage of very efficient
tools, specifically(Frobenius normal forms of matrices with entries in a field instead of
Fitting ideals of polynomial matrices. Then we turn to the general setting of Drinfeld $A$-modules where $A$ is a function ring.
(\S~\ref{subsec:comput-general}).

\subsection{The case $A = \Fq[T]$}
\label{subsec:comput-case-Fq[T]}

This case falls within the framework presented in
\S~\ref{subsubsec:case-FqT-abstract}. Reusing the same notations as in
\S~\ref{subsubsec:case-FqT-abstract}, we have $F = \Fq$, $\varphi = \phi_T$
and $V = \ker_K(u)$. The invariant factors and the Frobenius decomposition of
$\phi(\ker_K(u))$ (as an $\Fq[T]$-module) can be obtained by computing the
Frobenius normal form of the matrix of $\phi_T$ acting on $\ker_K(u)$ (as an
$\Fq$-vector space). We thus need two ingredients: fast computation of matrices
of Ore polynomials, and Frobenius normal forms (Lemma~\ref{lemma:cost-frob}).

\subsubsection{Computing the kernel of an isogeny}

\begin{lemma}
\label{lemma:cost-matrix}

  Let $f\in\Ktau$ be an Ore polynomial with $\tau$-degree $n$. Let $x_1,\dots,
  x_\ell$, with $\ell\leqslant d$, be elements of $K$. One can compute $f(x_1),\dots,
  f(x_\ell)$ for a cost of $O(n)$ arithmetic operations in $K$ and $O(\SMgeq(d)
  (\log d)^2)$ arithmetic and Frobenius operations in $K$. In particular, one
  can compute the matrix $M_f \in \Fq^{d\times d}$ of $f$ relative to $\E$ for
  the same cost.

\end{lemma}

\begin{proof}

  First, we reduce $f$ modulo $\tau^d - 1$, which can be done by shifting
  coefficients, for a cost of $O(n)$ arithmetic operations in $K$. Then the result follows from Lemma~\ref{lemma:cost-multipoint}. 
\end{proof}

Let $u$ be a morphism of Drinfeld $\Fq[T]$-modules over $K$ with domain $\phi$. Algorithm~\ref{algo:MorphismKernelInvariants} computes the invariant factors and a Frobenius decomposition of $\phi(\ker_K(u)$. Proposition~\ref{proposition:cost-keru} gives the asymptotic complexity of this algorithm.

%
%
%
%

\begin{algorithm}[ht]
\caption{\MorphismKernelInvariants}
\label{algo:MorphismKernelInvariants}

  \KwIn{A morphism $u$ from a Drinfeld $\Fq[T]$-module $\phi$ over the finite
  field $K$ equipped with the $\Fq$-basis $\E$}
  \KwOut{The invariant factors and a Frobenius decomposition of $\phi(\ker_K(u))$}

  \tcp{Compute $\ker_K(u)$}

  Compute the matrix $M_u \in \Fq^{d\times d}$ of $u$ relative to $\E$ \; \label{step:computing-Mu} 
  Compute an $\Fq$-basis $\B = (b_1, \dots, b_{d'})$ of $\ker_K M_u$ \; \label{step:compute-kerMu}
  \If{$\ker_K M_u = \{0\}$}{
    \KwRet{$(1 \in \Fq[T])$ and $(0 \in K)$}
  }
  Put the vectors of $\B$ in a matrix $N_u = [b_1,\cdots, b_{d'}] \in \Fq^{d\times d'}$ \;
  \tcp{Compute the action of $\phi_T$ on $\ker_K(u)$}

  Compute the coordinates $y_i$ of $\phi_T(b_i)$, $1\leqslant i\leqslant d'$,
  relative to $\E$ \;
  Put the vectors in a matrix $Y = [y_1,\dots, y_{d'}] \in \Fq^{d\times d'}$ \;
  Compute the unique matrix $X$ such that $N X = Y$ \; \label{step:compute-X}

  \tcp{Compute the Frobenius normal form of the action of $\phi_T$}

  Compute the Frobenius normal form $\Fnf_{X} \in \Fq^{d'\times d'}$ of $X$,
  together with a change-of-basis matrix $S \in \Fq^{d'\times d'}$ such that $S
  \Fnf_{X} S^{-1} = X$ \; \label{step:compute-frob-X}

  \tcp{Convert the vectors to the basis $\E$}

  Compute $S^{-1} \in \Fq^{d'\times d'}$ \;

  Write $\Fnf_X$ as a block diagonal matrix with entries $C_{d_1},\dots,
  C_{d_\ell}$ \;

  Compute the matrix $S' \coloneqq N S^{-1} \in \Fq^{d\times d'}$\; \label{step:Sprime}

  \tcp{Record and return the results}

  \For{$i$ from $1$ to $\ell$}{
    Write $d_i$ for the monic polynomials in $\Fq[T]$ corresponding to $C_{d_i}$ \;

    Write $x_i$ for the element of $K$ whose $F$-coordinates (relative to $\E$) is
    the column with index $\deg(d_1) + \cdots + \deg(d_{i-1})$ in $S'$ \;
  }

  \KwRet{$(d_1,\dots, d_\ell)$ and $(x_1,\dots, x_\ell)$.}

\end{algorithm}

\begin{proposition}
\label{proposition:cost-keru}

  Let $u$ be a morphism of Drinfeld $\Fq[T]$-modules whose domain is $\phi$.
  Algorithm~\ref{algo:MorphismKernelInvariants} computes the invariant factors
  and a Frobenius decomposition of $\ker u$ for a cost of $O(\stor{d})$ 
  arithmetic operations in $\Fq$, $O(r + \deg_\tau(u))$ arithmetic operations
  in $K$,
  and $O(\SMgeq(d)(\log d)^2)$ Frobenius and arithmetic operations in $K$.

\end{proposition}

\begin{proof}

  Step~\ref{step:computing-Mu}, computing $M_u$, costs $O(\deg(u))$ arithmetic operations in
  $K$ and $O(\SMgeq(d) (\log d)^2)$ arithmetic and Frobenius operations in $K$,
  as a direct application of Lemma~\ref{lemma:cost-matrix}. Computing an
  $\Fq$-basis of $\ker(M_u)$ (Step~\ref{step:compute-kerMu}) costs $O(\MM(d))$ operations in $\Fq$.

  The next costly step is the computation of the matrix~$X$ of $\phi_T$ as an
  $\Fq$-linear endomorphism of $\ker_K(u)$ (Step~\ref{step:compute-X}). The columns of this matrix are the
  coordinates of $\phi_T(b_1),\dots,\phi_T(b_{d'})$ relative to the basis $\B = (b_1,
  \dots, b_{d'})$ computed in Step~\ref{step:compute-kerMu}. Still by Lemma~\ref{lemma:cost-matrix}, the computation of
  $\phi_T(b_1),\dots,\phi_T(b_{d'})$ costs $O(r)$ arithmetic operations in $K$ and
  $O(\SMgeq(d)(\log d)^2)$ arithmetic and Frobenius operations in $K$. Now $X$ is uniquely determined by the equation
  $NX = Y$, which can be solved using standard linear
  algebra techniques, such as LUP decomposition of $N$ followed by
  forward and backward substitution, for a cost of $O(\MM(d))$
  arithmetic operations in $\Fq$ \cite{ibarra_generalization_1982}.

  The next costly step is the computation of the Frobenius normal form $\Fnf_X$
  of~$X$, together with the unimodular transformation matrix $S$
  (\S~\ref{step:compute-frob-X}). By Lemma~\ref{lemma:cost-frob}, this can be
  done for a cost of $O(\stor{d'})$ operations in~$\Fq$, which is contained in
  $O(\stor{d})$ operations in $\Fq$.

  Finally, we need to retrieve the Frobenius decomposition of
  $\phi(\ker_K(u))$. We cannot directly extract the vectors from
  $S^{-1}\in\Fq^{d'\times d'}$ (as explained in
  \S~\ref{subsubsec:case-FqT-abstract}), as the columns of $S^{-1}$ represent elements
  relative to the basis $\B$. Instead, we want the coordinates of
  these elements relative to the original basis $\E$. The coordinates of the
  columns of $S^{-1}$ relative to $\E$ are exactly the columns of $S'$
  (Step~\ref{step:Sprime}). The
  cost of computing $S'$ is also contained in $O(\MM(d))$ operations in $\Fq$.
\end{proof}

\begin{algorithm}[ht]
\caption{\ModuleOfPointsInvariants}
\label{algo:ModuleOfPointsInvariants}

  \KwIn{A Drinfeld $\Fq[T]$-module $\phi$ over the finite field $K$}
  \KwOut{The invariant factors and a Frobenius decomposition of $\phi(K)$}

  \KwRet{\MorphismKernelInvariants$(0_\phi)$} where $0_\phi$ is the zero morphism on $\phi$.

\end{algorithm}

\begin{corollary}
\label{cor:cost-points}

  Algorithm~\ref{algo:ModuleOfPointsInvariants} computes the invariant factors
  and a Frobenius decomposition of $\phi(\ker_K(u))$ for a cost of
  $O(\stor{d})$ arithmetic operations in $\Fq$, $O(r)$ arithmetic operations in
  $K$, and $O(\SMgeq(d)(\log d)^2)$ Frobenius and arithmetic operations in $K$.

\end{corollary}

\begin{remark}
\label{remark:zero-points}

  Implementing Algorithm~\ref{algo:ModuleOfPointsInvariants}
  (\ModuleOfPointsInvariants) without using
  Algorithm~\ref{algo:MorphismKernelInvariants} (\MorphismKernelInvariants)
  would only be marginally more efficient. The reason is that in the present
  description, many intermediate quantities simplify because the basis of
  $\ker_K(u)$ is $\B$ itself; we are then computing the Frobenius normal form
  of the matrix of $\phi_T$ relative to $\E$ itself. The matrix of $\phi_T$ can
  be pre-computed and used for other operations, for a cost of $O(r)$
  arithmetic operations in $K$ and $O(\SMgeq(d)(\log d)^2)$ arithmetic and
  Frobenius operations in $K$ (Lemma~\ref{lemma:cost-matrix}). The remaining
  operations cost $O(\stor{d})$ arithmetic operations in $K$
  (Lemma~\ref{lemma:cost-frob}).

\end{remark}

\subsubsection{Torsion submodules}
\label{subsubsec:special-torsion}

For this special case, where $u = \phi_a$ for some $a
\in \Fq[T]$, w can directly reuse the invariant factors $(d_i)_{1\leqslant i
\leqslant \ell}$ and the Frobenius decomposition $(x_i)_{1\leqslant i\leqslant
\ell}$ of $\phi(K)$, provided they have been computed (see
Algorithm~\ref{algo:ModuleOfPointsInvariants} and
Corollary~\ref{cor:cost-points}). Assuming that $x_i$ corresponds
to $d_i$ for each $i$, and following the notation given prior to Lemma~\ref{lemma:chain-gcd}, we write $\gamma_i
\coloneqq \gcd(a, d_i)$ and $\rho_i \coloneqq d_i / g_i$, for $1 \leqslant
i \leqslant \ell$. Then we have an isomorphism of $\Fq[T]$-modules
\[
  \phi(K)[a] \simeq \prod_{i=1}^\ell \Fq[T]/(\gamma_i).
\]
Furthermore, the elements $x_{a, i} = \phi_{\rho_i}(x_i)$ have order
$\gamma_i$. Considering only the set $I_a$ of indices $i$ such that
$\gamma_i \neq 1$, we have
\[
  \phi(K)[a] = \bigoplus_{i \in I_a} \Fq[T] x_{a, i}.
\]
In other words, $(\gamma_i)_{i \in I_a}$ and $(x_{a, i})_{i \in I_a}$
form the invariants factors and a Frobenius decomposition of
$\phi(K)[a]$. respectively. This leads to Algorithm~\ref{algo:TorsionFromModuleOfPoints}.


\begin{algorithm}[ht]
\caption{\TorsionFromModuleOfPoints}
\label{algo:TorsionFromModuleOfPoints}

  \KwIn{A polynomial $a \in A$, the invariant factors $(d_1,\dots, d_\ell)$ and
  a Frobenius decomposition $(x_1,\dots, x_\ell)$ of $\phi(K)$, where $\phi$ is
  a Drinfeld $\Fq[T]$-module over the finite field $K$}

  \KwOut{The invariant factors and a Frobenius decomposition of $\phi(K)[a]$}

  Compute $\gamma_i\coloneqq \gcd(a, d_i)$ and $\rho_i \coloneqq d_i/\gamma_i$, for all
  $1\leqslant i \leqslant \ell$, using Algorithm~\ref{algo:ChainGcd} \;

  Let $I_a$ be the set of indices $1\leqslant i \leqslant \ell$ such that
  $\gamma_i \neq 1$ \;

  Compute $x_{a, i} \coloneqq \phi_{\gamma_i}(x_i)$ for all $i \in I_a$.

  \KwRet{$(\gamma_i)_{i\in I_a}$ and $(x_{a, i})_{i \in I_a}$}

\end{algorithm}

%
%

For brevity, write $d_a \coloneqq \dim_\Fq \phi(K)[a]$, and $\ell_a
\coloneqq \# I_a$ for the number of invariant factors of $\phi(K)[a]$. Recall
Definition~\ref{def:ev} for the definition of $\EV$.

\begin{proposition}
\label{prop:inv-from-inv}

  Given the invariant factors and a Frobenius decomposition of $\phi(K)$, Algorithm~\ref{algo:TorsionFromModuleOfPoints} computes the invariant factors and a Frobenius decomposition of
  $\phi(K)[a]$ using for a cost
  of $O(\PM(\deg(a)) + \PM(d)\log d + \EV_d(\ell_a, d_a))$ operations in $\Fq$.

\end{proposition}

\begin{proof}

  The call to Algorithm~\ref{algo:ChainGcd} requires $O(\PM(\deg(a)) + \PG(d))$
  operations in~$\Fq$ by Lemma~\ref{lemma:chain-gcd}. The last step requires $O(\EV_d(\ell_a, d_a))$ operations in $K$,
  where $d_a = \dim_\Fq\phi(K)[a])$ and $\ell_a = \# I_a$ is the number of
  invariant factors. Asymptotically, $\PG(d) \in O(d^2)$ and $\EV_d(\ell_a,
  d_a)$ is at most $O(d^2)$, so the total cost of Algorithm~\ref{algo:TorsionFromModuleOfPoints} is $O(\PM(\deg(a)) + \EV_d(\ell_a,
  d_a))$.
\end{proof}

\subsubsection{Explicit complexity results}
\label{subsec:discuss}

We determine the cost of some of our algorithms, depending on different
primitives for Ore polynomials. We start with
Algorithm~\ref{algo:ModuleOfPointsInvariants}. This algorithm essentially has
two steps (Remark~\ref{remark:zero-points}): computing the matrix $\MphiT$ of
$\phi_T$ relative to the $\Fq$-basis $\E = \{ \varepsilon_1, \ldots , \varepsilon_d\}$ of $K$ and computing its Frobenius
decomposition. We first ascertain the cost of computing $\MphiT$. We recall that as
$K$ is a finite field with $\Fq$-dimension $d$, an arithmetic operation in
$K$ can be computed for a cost of $\Otilde(d \log q)$ bit operations
\cite[Chapter~2]{von_zur_gathen_modern_2013}.

\begin{enumerate}

  \item \label{item:naive} The most na\"ive method consists of sequentially
    evaluating $\phi_T(\varepsilon_1), \dots, \phi_T(\varepsilon_d)$. Counting
    $r$ Frobenius applications for an evaluation and $O(\log q)$ operations in
    $K$ for a Frobenius application using square \& multiply, we obtain
    $\Otilde(rd^2 (\log q)^2)$ bit operations to compute the matrix of $\phi_T$.

  \item If we use multipoint evaluation (Algorithm~\ref{algo:multieval} and
    Lemma~\ref{lemma:cost-matrix}) with $\SMF(d) = d^2$ and square \&
    multiply, the cost of computing $\MphiT$ becomes $\Otilde(r + d^3 (\log
    q)^2)$. In other words, fast multipoint evaluation is only possible with
    fast primitives for Ore polynomial multiplication.

  \item \label{item:cl} If $\Otilde(\SMgeq(d))$ operations in $K$ account for
    $\Otilde(d^{\frac {9 - \omega}{5 - \omega}} \log q)$ bit operations as in
    \cite{caruso_fast_2017}\footnote{The value of $\SMgeq$ in the original
    article contains a typo that was corrected in
    \cite[\S~1.2.3]{caruso_algorithms_2026}.}, then we can compute $M_{\phi_T}$
    in $\Otilde(rd\log q + d^{\frac {9 - \omega}{5 - \omega}} (\log q)^2)$ bit
    operations. If $r$ is in $O(d)$, this is better than the na\"ive method of
    item~\ref{item:naive}; if on the other hand $d$ grows while $r$ remains
    constant, then the na\"ive method yields a better complexity in $d$.

  \item If we use the Euclidean algorithm
    \cite[Algorithm~3]{leudiere_computing_2024} in
    Algorithms~\ref{algo:tree} and \ref{algo:multieval}, then computing $M_{\phi_T}$ is asymptotically no faster than the na\"ive method, as the first llcm computation alone 
    requires a quadratic number of Frobenius and arithmetic operations.

  \item \label{item:ku} Using the Kedlaya-Umans algorithm
    \cite{kedlaya_fast_2011} for Frobenius applications instead of square \& multiply in the na\"ive method
    (item \ref{item:naive}),
    Lemma~\ref{lemma:cost-matrix} gives a cost $\Otilde(d(\log q)^2) + (rd^2\log
    q)^{1+o(1)}$ bit operations.

\end{enumerate}
This shows that the three most advantageous methods for computing $\MphiT$ are
the na\"ive method (item~\ref{item:naive}), multipoint evaluation using
the fast multiplication algorithm (and the complexity model) of Caruso and Le
Borgne (item~\ref{item:cl}), and multipoint evaluation using the Kedlaya-Umans
algorithm for applications of the Frobenius endomorphism (item~\ref{item:ku}).

\begin{enumerate}

  \item[(\ref{item:naive})] If we use the na\"ive method for the
    computation of $\MphiT$ (Item~\ref{item:naive}), then the cost of
    Algorithm~\ref{algo:ModuleOfPointsInvariants} is $\Otilde(r d^2(\log q)^2 +
    \MM(d)\log q)$ bit operations.

  \item[(\ref{item:cl})] If we use multipoint evaluation
    (Algorithm~\ref{algo:multieval}) and the model and primitives of
    Caruso and Le Borgne for computing $\MphiT$, the cost of
    Algorithm~\ref{algo:ModuleOfPointsInvariants} is $\Otilde(rd\log q +
    d^{\frac {9 - \omega}{5 - \omega}} (\log q)^2)$ bit operations. With
    respect to the variable $d$, the cost of computing $\MphiT$ is therefore
    dominant.

  \item[(\ref{item:ku})] If we use multipoint evaluation
    (Algorithm~\ref{algo:multieval}) and the model and primitives of
    Kedlaya and Umans (item~\ref{item:ku}) in Lemma~\ref{lemma:cost-matrix},
    the total cost of Algorithm~\ref{algo:ModuleOfPointsInvariants} is
    $\Otilde(d(\log q)^2 + \MM(d) \log q) + (rd^2\log q)^{1+o(1)}$ bit
    operations.

\end{enumerate}
More generally, we can always assume that $\SMA(d) \in \MM(d)$ (the na\"ive
method for multiplying two Ore polynomials yields $\SMA(d) = d^2$). If we fix
$\log q$ and use $\Otilde$ for clarity, the total cost of
Algorithm~\ref{algo:ModuleOfPointsInvariants}, as per
Corollary~\ref{cor:cost-points}, is $\Otilde(r + \MM(d))$ arithmetic operations
and $\Otilde(\SMF(d))$ applications of the Frobenius endomorphism. In a model
where the cost of applying Frobenius is no higher than the cost of other
arithmetic operations, this is $\Otilde(r + \MM(d))$ arithmetic operations, and
the computation of the Frobenius normal form dominates.

Next, we investigate the cost of computing the invariant factors and a Frobenius
decomposition of $\phi(K)[a]$, $a\in \Fq[T]$, knowing the invariant factors and
a Frobenius decomposition of $\phi(K)$. As before, write $d_a \coloneqq
\dim_\Fq \phi(K)[a]$, and $\ell_a \coloneqq \# I_a$ for the number of invariant
factors of $\phi(K)[a]$. As discussed in \S~\ref{subsec:div-chains}, the cost
of Algorithm~\ref{algo:TorsionFromModuleOfPoints} (see
Proposition~\ref{prop:inv-from-inv}) is either $O(\PM(\deg(a)) + \PM(d)\log d +
d_a d^2)$ operations in $\Fq$ using a direct method, or $O(\PM(\deg(a)) +
\PM(d)\log d + \ell_a \MM(d))$ operations in $\Fq$ using the primitive of
Storjohann for polynomial-matrix evaluation~\cite{storjohann_algorithms_2000}.
The first method is preferred when $d_a$ is small with respect to~$d$;
otherwise, the second method should be used. In particular, we see that if
the rank is large relative to $\deg(a)$ or $d_a$ is small relative
to $d$, then it is cheaper to compute the invariant factors and Frobenius
decomposition of $\phi(K)[a]$ by reusing those of $\phi(K)$.

\subsection{The general case}
\label{subsec:comput-general}

For the general case, we apply the results of \S~\ref{subsection:general-case}.
Let~$A$ be a function ring. Then $A$ is a finitely generated $\Fq$-algebra. The curve
$C$ and the function ring $A$ can be represented in many ways. Here, we assume
that $A$ is given as the quotient $\Fq[T_1,\dots, T_m]/(P)$, for some
polynomial $P$, and we write $g_i \coloneqq T_i \mod P$, $1\leqslant i
\leqslant n$. Let $\phi$ be a Drinfeld $A$-module over $(K, \gamma)$ and $u$
a morphism whose domain is $\phi$. Our goal is to compute the invariant
factors of $\phi(\ker_K(u))$. 
We follow \S~\ref{subsection:general-case}, with $V = \ker_K(u)$. For each $1 \leqslant
k \leqslant n$, let $M_i$ be the matrix (w.r.t the fixed basis $\E$) of the
$\Fq$-linear endomorphism $\phi_{g_i}$ acting on $\ker_K(u)$. Let $d_u
\coloneqq \dim_\Fq(\ker_K(u))$. Then we have an isomorphism of $A$-modules
\[
  \phi(\ker_K(u))\simeq A^{d_u} / \Im(M_\chi),
\]
where $M_\chi$ is the vertical block matrix of blocks $(g_k \Id -
M_k)_{1\leqslant k \leqslant n}$ as in Equation~\eqref{eq:stacking}. Computing
the invariant factors of $\phi(\ker_K(u))$ then amounts to a simple sequence of
steps:
\begin{enumerate}

  \item Computing the matrices $M_1,\dots, M_m$, and the matrix $M_\chi$;
  \item Computing the Fitting ideals $\Fitt_1(M),\dots,\Fitt_d(M)$\;
  \item Computing the ideals $\d_i \coloneqq \Fitt_{d-i}(M)/\Fitt_{d+i-1}(M)$,
    for $1\leqslant i \leqslant d$.
\end{enumerate}
By Lemma~\ref{lemma:inv-fit}, the ideals $(\d_i)_{1\leqslant i\leqslant d}$
are the invariant factors of $\phi(\ker_K(u))$.





%
%
%
%
%
%
%
%
%

\section{Deciding when the torsion is rational}
\label{sec:deciding-rational}

We conclude with a problem closely related to the computation
of rational torsion submodules, namely determining all $a \in \Fq[T]$ for which $\phi[a]$
is rational, where $\phi$ is a Drinfeld $\Fq[T]$-module over a finite field
$K/\Fq$ with $\Fq$-dimension $d$. While all the points of $\phi(K)$ are torsion
points (they are of $\chi_\phi(1)$-torsion, where $\chi$ is the \emph{characteristic
polynomial of the Frobenius endomorphism}), the existence of a non-constant $a
\in \Fq[T]$ for which the $a$-torsion is rational is not guaranteed.

\subsection{Anderson motives}
\label{subsubsec:anderson-motives}

First, we briefly introduce \emph{Anderson $\Fq[T]$-motives}; see
\cite{anderson_t-motives_1986, van_der_heiden_weil_2004,
papikian_drinfeld_2023, armana_computational_2025} for a general introduction.

\begin{definition}

  The \emph{Anderson $\Fq[T]$-motive} attached
  to a Drinfeld module $\phi$ over~$K$, denoted $\M(\phi)$, is the
  $K[T]$-module defined by
  \FunctionNoname
    {K[T] \times \Ktau}
    {\Ktau}
    {\left(\sum_{i=0}^n \lambda_i T^i, f(\tau) \right)}
    {\sum_{i=0}^n \lambda_i f(\tau) \phi_{T^i}.}

\end{definition}

Importantly, $\M(\phi)$ is a free $K[T]$-module with
\emph{canonical basis} $\{1, \tau, \dots, \tau^{r-1}\}$. Moreover, $\M$ defines a
contravariant functor, as for any morphism of Drinfeld modules $u:\phi\to\psi$,
there exists a morphism of $K[T]$-modules $\M(u):\M(\psi) \to \M(\phi)$,
defined by $\M(u)(f) = fu$. The functor $\M$ is fully faithful, and we can
therefore represent Ore polynomials by tuples of $r$ coefficients in $K[T]$.
This representation is easily computed
\cite[Algorithm~1]{caruso_algorithms_2026}.

\subsection{Deciding when the torsion is rational}
Current methods for deciding rationality only handle a fixed $a \in \Fq[T]$.

\begin{itemize}

  \item For any $s$, write $K_s$ for the degree $s$ extension of $K$. To determine is the minimal field of definition (\emph{i.e.} the splitting field)
    of the $a$-torsion, one can use Anderson motives
    (\S~\ref{subsubsec:anderson-motives}). Indeed, $\phi(K_s)[a] = \phi[a]$ if
    and only if the Frobenius endomorphism of $K$ is the identity on $\phi[a]$.
    Using the correspondence between Drinfeld modules and Anderson motives
    via the functor $\M$, this is equivalent to $\M(\tau^{sd})$ being
    the identity on the quotient module $\M(\phi)/\M(\phi)\phi_a$ (see
    \cite[\S~3.6]{papikian_drinfeld_2023} for details). We can effectively
    write $\M(\tau^d)$ as a matrix $M \in (\Fq[T]/(a))^{r\times r}$. The
    minimal $s$ such that $\phi(K_s)[a] = \phi[a]$ is then the minimal $s$ such
    that $M^s$ is the identity. However, computing $s$ seems computationally
    hard. For example, when $\Fq[T]/(a)$ is a field, finding the idempotency
    order of a matrix requires finding the multiplicative order elements in
    $\Fq[T]/(a)$. As of now, the best methods to effect this require factoring an
    integer of the form $q^n - 1$ for some $n$.

  \item If $a$ is given, it is easy to know if the
    $a$-torsion is rational. If $h$ is the $\tau$-valuation of $\phi_a$, then
    the $a$-torsion is rational if and only if $\phi_a$ right-divides
    $\tau^h(\tau^d - 1)$.

\end{itemize}

These two methods determine rationality of the torsion for
individual $a \in \Fq[T]$, but are unable to find all $a \in \Fq[T]$ with rational $a$-torsion. We introduce a method that accomplishes this, beginning with the separable case. Consider the Anderson motive $\M(\phi)$ with its canonical
$K[T]$-basis $\{1, \dots, \tau^{r-1}\}$. As $K$ has finite
$\Fq$-dimension $d$, the Anderson motive $\M(\phi)$ is free over $\Fq[T]$.
Let $p\in\Fq[T]$ be the monic generator of the kernel of the $\Fq$-algebra morphism $\gamma: \Fq[T] \rightarrow K$. Let $a$ be
coprime to $p$, so $\phi_a$ is separable. Then the $a$-torsion is
$K$-rational if and only if $\phi_a$ right-divides $\tau^d - 1$. As the
$\Fq[T]$-coordinates of $\phi_a$ are simply $(a, 0,\dots)$, we have the
following proposition.

\begin{proposition}
\label{prop:decide-rational-torsion}

  Let $g_\phi$ be the gcd of the $\Fq[T]$-coordinates of $\tau^d - 1 \in
  \M(\phi)$. If $a$ is coprime to $p$, then the $a$-torsion of $\phi$ is rational if and only
  if $a$ divides $g_\phi$.

\end{proposition}

In other words, the element $g_\phi$ is the lcm of all $a\in\Fq[T]$ coprime to
$p$ for which the $a$-torsion is rational. This invariant can also be recovered
using the invariant factors of $\phi(K)$, and the fact that $\phi[a]\simeq
(\Fq[T]/(a))^r$, but 
we present a more direct and efficient
approach.

Note that Proposition~\ref{prop:decide-rational-torsion} does not hold if $a$
is a multiple of $p$. Indeed, $\phi_a$ is inseparable in this case, and so would
be all its multiples. However, $\tau^d - 1$ is separable. To determine the maximum
$s$ such that the $p^s$-torsion is rational, one can however proceed as
follows.

\begin{enumerate}

  \item Compute the product $n_\phi$ of the invariant factors of $\phi(K)$. One
    can obtain $n_\phi$ directly from the characteristic polynomial of
    the Frobenius endomorphism~$\chi_\phi$, as $n_\phi = \chi_\phi(1)$ (see
    \cite[Appendix~A]{caruso_algorithms_2026} for a review of methods for computing $\chi_\phi$). 
    Indeed, any $a\in\Fq[T]$ such that
    $\phi[a]$ rational, is a divisor of $n_\phi$.

  \item Compute the $p$-adic valuation $v$ of $n_\phi$. This can be done for a
    cost of $O(\PM(\deg(n_\phi))\log(\deg(n_\phi)/\deg(p)))$ operations in
    $\Fq$ using the \emph{divide and conquer} method of \emph{radix conversion}
    \cite[\S~9.2]{von_zur_gathen_modern_2013}.

  \item The desired integer $s$ is such that $s \leqslant v$. One can find it
    by trying all possibilities or by implementing a more sophisticated search
    algorithm.

\end{enumerate}

\begin{remark}

  The cost of computing $g_\phi$ only depends on the size of the input and not
  on the size of the output. Said cost is polynomial. By contrast, we do not
  know any efficient method to compute the analogue of $g_\phi$ for an elliptic
  curve over a finite field.

\end{remark}

The computation of the invariant $g_\phi$ from $n_\phi$, along with the torsion
rationality test, were also implemented in SageMath and demonstrated in the
\emph{Jupyter} notebook mentioned in the introduction.

\section*{Conclusion}


All the algorithms presented herein are deterministic and work well in full generality. Nevertheless, there may be room for improvement.
\begin{enumerate}

  \item In light of Algorithm~\ref{algo:TorsionFromModuleOfPoints}, one might ask whether it could be profitable to compute the invariant factors and a Frobenius
    decomposition of $\phi(\ker_K(u))$, for some Drinfeld module $\phi$ and
    isogeny $u$, by reusing the invariant factors and Frobenius decomposition
    of $\phi(K)$. One would first need to find a matrix representation of $u$ as an $\Fq[T]$-morphism (\emph{e.g.} a presentation matrix), and not
    simply as an $\Fq$-morphism. The cost of producing such a representation
    is at least $O(\MM(d))$ operations in $d$, which is not better (in
    $d$) than the estimate of Proposition~\ref{prop:inv-from-inv}. Ultimately, representing
    isogenies as Ore polynomials is straightforward, but not necessarily the
    best representation for computations.

  \item In Algorithm~\ref{algo:TorsionFromModuleOfPoints}, we compute
    quantities of the form $P(M)x$, where $P$ is a polynomial of $\Fq[T]$ of
    degree $n \leqslant d$, $M$ is a $d$-by-$d$ matrix with entries in $\Fq$,
    and $x$ is a vector. The assumption $n \leqslant d$ is critical for obtaining a
    better complexity (using Storjohann's method
    \cite[Chapter~9]{storjohann_algorithms_2000}), as without it the best
    method to compute $P(M)$ is that of Paterson and Stockmeyer
    \cite{paterson_number_1973}, for a cost of $O(\sqrt{n} \MM(d) + n d^2)$
    operations in $\Fq$. The Paterson-Stockmeyer method is optimal in the
    general case. This begs the question of adapting the Paterson-Stockmeyer method (which
    is itself a direct adaptation of the Horner method \cite{horner}) for
    computing $P(M)x$ without computing $P(M)$. It seems that a
    fundamental obstruction is the absence of a relevant vector product.
    Current workarounds involve a substantial number of precomputations, resulting in an
    identical or worse complexity than Storjohann's method. This raises the
    question whether there exists a method for computing $P(M)x$ (in full
    generality) whose cost is
    $O(nd + \sqrt{n}d^2)$ (cost optimal in $d$) operations in $\Fq$.

  \item Algorithm~\ref{algo:TorsionFromModuleOfPoints} also requires
    computations related to divisibility chains. Specifically, we need to
    compute elements of the form $\phi_{\gamma_i}(x_i)$, where the $\gamma_i$
    form a divisibility chain. We have not found a way to leverage said
    divisibility chain in a way that significantly improves the worst-case
    complexity, as for any two $i, j$ with $i<j$, knowing
    $\phi_{\gamma_i}(x_i)$, $\gamma_j/\gamma_i$ and $x_j$ does not \emph{a
    priori} help to find $\phi_{\gamma_j}(x_j)$. Further investigation in these
    problems, irrespective of their application to Drinfeld modules, is
    warranted. In particular, probabilistic methods could help in some cases,
    \emph{e.g.} when the invariant factors are smooth.

    \item The method of \ref{subsec:comput-general} also leaves open questions. The assumption that $A$ is given as a quotient of $\Fq[T_1,\dots, T_m]$ enables the use of Gröbner
  bases techniques for the computation of the invariant factors of $\ker_K(u)$. Exploring and implementing this direction may be worthwhile.

\end{enumerate}

Finally, we reflect on the use of finite fields. This is
the setting for most practical applications, such as computer algebra
\cite{doliskani_drinfeld_2021,bastioni_characteristic_2025} and coding theory
\cite{bastioni_optimal_2024, micheli_rank_2026}). If $\phi$ is a Drinfeld module over a finite field $K$, then $\phi(K)$ is finitely generated, both when $A = \Fq[T]$ and
for general function fields. We can leverage the structure of $A$ as a finitely
generated $\Fq$-algebra and encode the action of the generators using
matrices. This is not possible for elliptic curves, where the role of $A$
is played by $\ZZ$. If the base field $K$ is a function field, then the module
$\phi(K)$ is not finitely generated and its non-torsion part is free with
infinite rank $\aleph_0$ \cite{poonen_local_1995}. As pointed out to us by
Caruso, the torsion part is known to be finite
\cite[Theorem~6.4.3]{papikian_drinfeld_2023}.
Anderson motives would likely be an important tool in the computation of these
torsion points. Even in the infinite
case (but with a suitable changes of variables), they facilitate a representation of an isogeny by a finite square polynomial matrix whose cokernel is isomorphic
to the kernel of the isogeny. The invariant factors of the cokernel can simply be
recovered by computing the Smith normal form of the matrix. For finite fields,
this strategy is often less efficient compared to our approach due to the costly
manipulation of polynomial matrices. However, in the general case, Anderson motives
simplify the problem by reducing it to classical computer algebra. There exists a monic separable
isogeny with minimal degree whose kernel is the torsion part of $\phi(K)$; finding this isogeny is a task for future work.

\section*{Acknowledgments}

We warmly thank Xavier Caruso for insightful discussions. Antoine Leudière is
supported in part by the Pacific Institute for the Mathematical Sciences.
Renate Scheidler is supported by the Natural Sciences and Engineering Research
Council of Canada (NSERC), funding RGPIN-2025-03992.

\printbibliography

\end{document}